\documentclass[a4paper, 11pt,oneside]{article}
\usepackage{amsmath}
\usepackage{amsthm}
\usepackage{amssymb}
\usepackage{bbm}
\usepackage{graphicx}
\usepackage{authblk}
\usepackage{tikz}
\usepackage{caption}
\usetikzlibrary{shapes,positioning}
\usetikzlibrary{automata}
\usetikzlibrary{arrows}

\theoremstyle{plain}
\newtheorem{theorem}{Theorem}
\newtheorem{lemma}{Lemma}
\newtheorem{corollary}{Corollary}
\newtheorem{proposition}{Proposition}

\newtheorem{definition}{Definition}

\newtheorem{conjecture}{Conjecture}

\theoremstyle{remark}
\newtheorem{remark}{Remark}

\newcommand{\spl}{\mbox{\sf Split}}
\newcommand{\id}{\mbox{\it id}}

\title{Synchronizing non-deterministic finite automata}
\author[1]{Henk Don}
\author[2,3]{Hans Zantema}
\affil[1]{\small Department of Mathematics, Vrije Universiteit Amsterdam, De Boelelaan 1081a, 
	1081 HV Amsterdam, 
	The Netherlands, email: {\tt h.don@vu.nl} }
\affil[2]{Department of Computer Science, TU Eindhoven, P.O.\ Box 513,\\
		5600 MB Eindhoven, The Netherlands, 
		email: {\tt h.zantema@tue.nl}}
\affil[3]{Radboud University Nijmegen, P.O.\ Box 9010, \\ 6500 GL Nijmegen, The Netherlands}

\begin{document}
\maketitle

\begin{abstract}
In this paper, we show that every D3-directing CNFA can be mapped uniquely to a DFA with the same synchronizing word length. This implies that \v{C}ern\'y's conjecture generalizes to CNFAs and that the general upper bound for the length of a shortest D3-directing word is equal to the Pin-Frankl bound for DFAs. As a second consequence, for several classes of CNFAs sharper bounds are established. Finally, our results allow us to detect all critical CNFAs on at most 6 states. It turns out that only very few critical CNFAs exist. 
\end{abstract}

\section{Introduction and preliminaries}

In this paper we study synchronization of non-deterministic finite automata (NFAs). As is the case for deterministic finite automata (DFAs), symbols define functions on the state set $Q$. However, in an NFA symbols are allowed to send a state to a subset of $Q$, rather than to a single state. An NFA is called complete if these subsets are non-empty. This basically says that in every state, every symbol has at least one out-going edge. Formally, a {\em complete non-deterministic finite automaton (CNFA)} $\mathcal{A}$ over a finite alphabet $\Sigma$ consists of a finite set
$Q$ of states and a map $\delta: Q \times \Sigma \to 2^Q\setminus \{\emptyset\}$. We denote the number of states by $|\mathcal{A}|$ or by $|Q|$. 

A DFA is called synchronizing if there exists a word that sends every state to the same fixed state. In 1964 \v{C}ern\'y \cite{C64} conjectured that a synchronizing DFA on $n$ states always admits a synchronizing (or directing, reset) word of length at most $(n-1)^2$. He gave a sequence $C_n$ of DFAs in which the shortest synchronizing word attains this bound. In this paper, we denote the maximal length of a shortest synchronizing word in an $n$-state DFA by $d(n)$. The best known bounds for $d(n)$ are
\begin{equation}
(n-1)^2 \leq d(n) \leq \frac{n^3-n}{6}.
\end{equation} 
For a proof of the upper bound, we refer to \cite{pin}. A DFA on $n$ states is 
{\em critical} if its shortest synchronizing word has length $(n-1)^2$; it is 
{\em super-critical} if its shortest synchronizing word has length $> (n-1)^2$. So \v{C}ern\'y's 
conjecture states that no super-critical DFAs exist. It turns out that there are not too many 
critical DFAs. Investigation of all critical DFAs with less than 7 states but unrestricted alphabet was recently completed \cite{DZ17,BDZ17}. 
DFAs without copies of the same symbol and without the identity are called \emph{basic}. For
$n=3,4,5,6$, only 31 basic critical DFAs exist up to isomorphism. So critical DFAs are very infrequent, 
as the total number of basic DFAs on $n$ states is $2^{n^n-1}$, including isomorphisms. For $n\geq 7$, 
the only known examples are from \v{C}ern\'y's sequence.

For $S\subseteq Q$ and $w\in\Sigma^*$, let $Sw$ be the set of all states where one can end when starting in some state $q\in S$ and reading the symbols in $w$ consecutively. Write $qw$ for $\left\{q\right\}w$. Formal definitions will be given in Section \ref{sec:preliminaries}. A DFA is synchronizing if there exists $w\in \Sigma^*$ and $q_s\in Q$ such that $qw = q_s$ for all $q\in Q$. There are several ways to generalize this concept of synchronization to CNFAs, see \cite{imreh}. In this paper, we study CNFAs known in the literature as D3-directing. This notion is defined as follows:

\begin{definition} A CNFA $(Q,\Sigma,\delta)$ is called D3-directing if there exists a word $w\in\Sigma^*$ and a state $q_s$ such that $q_s\in qw$ for all $q\in Q$. The word $w$ is called a D3-directing word. 
\end{definition}

An example of a D3-directing CNFA is depicted below. There exist several D3-directing words of length four, but no shorter ones. An example is $w = baba$, which gives $1w = \left\{1,3\right\}$, $2w = \left\{1,2\right\}$ and $3w = \left\{1,2,3\right\}$. The synchronizing state for this word is 1. Another D3-directing word is $v = aabb$, for which $1v =3v = \left\{1,2,3\right\}$ and $2v = 2$. Here the synchronizing state is 2.   

\begin{center}
	\begin{tikzpicture}[-latex',node distance =2 cm and 2cm ,on grid ,
	semithick , state/.style ={ circle ,top color =white , bottom
		color = white!20 , draw, black , text=black , minimum width =.5
		cm}] 
	
	\node[state] (1) {1}; 
	\node[state] (2) [below left =of 1] {2};
	\node[state] (3) [below right =of 1] {3}; 
	\path (1) edge [bend right = 15] node[above left] {b} (2); 
	\path (1) edge [bend right = 15] node[below left] {a} (3);
	\path (2) edge [bend right = 15] node[below right] {b} (1); 
	\path (3) edge [bend right = 15] node[above right] {a} (1);
	\path (3) edge node[below] {b} (2); 
	\path (1) edge [loop above,looseness=8, in = 60, out = 120] node[above] {a} (1); 
	\path (2) edge [loop left,looseness=8, in = 180, out = 240] node[left] {a} (2);
	\path (3) edge [loop right,looseness=8, in = 300, out = 0] node[right] {b} (3);
	\end{tikzpicture}
\end{center}

A word is D3-directing if starting in any state $q$, there exists a path labelled by $w$ that ends in $q_s$. For DFAs this notion coincides with a synchronizing word. If a CNFA $\mathcal{A}$ is D3-directing, a natural question is to find the length of a shortest D3-directing word. We denote this length by $d_3(\mathcal{A})$. Furthermore, we denote by $cd_3(n)$ the worst case, i.e. we let CDir(3) be the collection of all D3-directing CNFAs and define
\begin{equation}
cd_3(n)  = \max\left\{d_3(\mathcal{A}):\mathcal{A}\in {\rm CDir(3)},|\mathcal{A}|=n\right\}.
\end{equation}

In  \cite{imreh} it is shown that for all $n\geq 1$, 
\begin{equation}
(n-1)^2 \leq cd_3(n) \leq \frac12 n(n-1)(n-2)+1.
\end{equation}
The lower bound follows from the fact that every DFA is also a CNFA and that for DFAs the notions of synchronization and D3-directability coincide. As far as we are aware, these bounds are still the sharpest known for CNFAs, although sharper results were recently obtained for the essentially equivalent problem of bounding lengths of column-primitive products of matrices \cite{CHJ15}. Analogous to DFAs, a D3-directing CNFA is called critical if its shortest D3-directing word has length $(n-1)^2$, and super-critical if it has length $>(n-1)^2$. 
In the current paper, we will prove that in fact $cd_3(n) = d(n)$, which immediately sharpens the upper bound for $cd_3(n)$ to $(n^3-n)/6$. Our result also implies that \v{C}ern\'y's conjecture is equivalent to the following:
\begin{conjecture}\label{conjecture1}
Every D3-directing CNFA with $n$ states admits a D3-directing word of length at most $(n-1)^2$.
\end{conjecture} 
The main ingredient to prove that $cd_3(n) = d(n)$ is a splitting transformation $\spl$ that maps a CNFA to a DFA. Every D3-directing CNFA $\mathcal{A}$ is transformed into a synchronizing DFA $\spl\mathcal{(A)}$, preserving the shortest D3-directing word length. For several classes of DFAs, the \v{C}ern\'y conjecture has been established, or sharper bounds than the general bounds have been proven, see for example \cite{ananichev, beal, don, dubuc, E90, kari, volkov2}. If $\spl\mathcal{(A)}$ satisfies the properties for one of these classes, then the sharper results for $\spl\mathcal{(A)}$ also apply to the CNFA $\mathcal{A}$. This observation gives rise to generalize several properties of DFAs into notions for CNFAs and to check if these generalized properties are preserved under $\spl$. In this way, we derive sharper upper bounds on the maximal D3-directing word length for several classes of CNFAs.  

Finally, in this paper we search for examples of critical D3-directing CNFAs. Note that the number
of CNFAs without identical symbols on $n$ states is huge, namely $2^{(2^n-1)^n}$ when we include
isomorphisms. Therefore an exhaustive search is problematic, even for small $n$. However, since we
know that every critical CNFA can be transformed into a critical DFA, we can try to find critical
examples by reversing the transformation $\spl$. Since all critical DFAs on $\leq 6$ states are known, 
this approach allows us to identify all critical CNFAs on $\leq 6$ states. Applying this strategy to 
the other known critical DFAs, the only critical CNFAs we find are small modifications of 
\v{C}ern\'y's sequence. 

\subsection{Preliminaries}\label{sec:preliminaries}

In this section we present our formal definitions and notation which will be slightly different from the
traditional notation, as we avoid the use of the transition function.  A symbol (or letter, label) $a$ in a
CNFA will be a function $a:Q\rightarrow 2^Q\setminus\{\emptyset\}$, and we denote $a(q)$ by $qa$. A symbol
extends (denoting the extension by $a$ as well) to a function $a:2^Q\rightarrow 2^Q\setminus\{\emptyset\}$ by
$Sa = \bigcup_{q\in S}qa$. The set of all letters on $Q$ that can be obtained in this way is denoted $T(Q)$,
which is a strict subset of the set of all functions from $2^Q$ to $2^Q\setminus\{\emptyset\}$.  The set of all possible symbols in a DFA on its turn is a subset of $T(Q)$:
$$
T^d(Q) = \left\{a\in T(Q):\forall q\in Q\ |qa|=1\right\}.
$$
A CNFA $\mathcal{A}$ is defined to be a pair $(Q,\Sigma)$, where $\Sigma\subseteq T(Q)$. Similarly a DFA is a pair $(Q,\Sigma)$ with $\Sigma\subseteq T^d(Q)$. Note that these definitions do not allow for two symbols that act exactly in the same way: if $a$ is a possible symbol, then either $a\in\Sigma$ or $a\not\in\Sigma$.

A symbol $a\in T(Q)$ induces a directed graph $G_a$ with vertex set $Q$ and can therefore be viewed as a subset of $Q\times Q$: 
$$
a = \left\{(q,p):q\in Q, p\in qa\right\},
$$ 
so $a$ is identified with the set of all edges in $G_a$.    
This point of view is used to define set relations and operations like inclusion and union on $T(Q)$. For example, if $a,b\in T(Q)$, then
$$
a\cup b = \left\{(q,p):q\in Q, p\in qa\cup qb\right\}.
$$

Suppose $\mathcal{A} = (Q,\Sigma)$ is a CNFA and $a,b\in\Sigma$ are such that $a\subseteq b$. If $\mathcal{A}$ is D3-directing, then the automaton $(Q,\Sigma\setminus\left\{a\right\})$ is D3-directing as well with the same shortest synchronizing word length. Also the identity symbol has no influence on synchronization. Therefore, a CNFA is called \emph{basic} if it has no identity symbol and no symbol is contained in another one. For DFAs this coincides with the existing notion of basic.

If $\mathcal{A} = (Q,\Sigma)$ and $\mathcal{B} = (Q,\Gamma)$ are CNFAs, we say that $\mathcal{B}$ is contained in $\mathcal{A}$ and write $\mathcal{B}\subseteq\mathcal{A}$ if for all $b\in\Gamma$ there exists $a\in\Sigma$ such that $b\subseteq a$. Alternatively, we say that $\mathcal{A}$ is an extension of $\mathcal{B}$. If $\mathcal{B}\subseteq\mathcal{A}$ and $\mathcal{B}\neq\mathcal{A}$, we say that $\mathcal{A}$ strictly contains (or is a strict extension of) $\mathcal{B}$.  A critical CNFA is {\em minimal} if it is not the strict extension of another critical CNFA; it is {\em maximal} if it does not admit a basic critical strict extension.

Finally, for $w =w_1\ldots w_k\in \Sigma^*$ and $S \subseteq Q$, define $Sw$ inductively by $S \epsilon = S$ and $S w = (Sw_1\ldots w_{k-1})w_k$. So a word also is a function on $2^Q$, being the composition of the transformations by each of its letters. Therefore also the transition monoid $\Sigma^*$ is contained in $T(Q)$. 

\section{Transforming a CNFA into a DFA, preserving D3-directing word length}

In this section we present the transformation $\spl$ and explore some of its properties. We note that similar but less explicit ideas were recently used in \cite{jungers,gusev} to give bounds on the length of a positive product in a primitive set of matrices. We start by introducing a parametrized version of our transformation:
\begin{definition}\label{def:parsplit}
	Let $\mathcal{A}=(Q,\Sigma)$ be a CNFA. Fix $q_{\textit{split}}\in Q, a\in\Sigma$ and denote the set $q_{\textit{split}}a$ by $\left\{q_1,\ldots,q_m \right\}$. Define new symbols $a_1,\ldots, a_m$ on $Q$ as follows:
	$$
	q_{\textit{split}}a_i := q_i, \quad\textrm{and}\quad qa_i := qa\quad\textrm{for\ }q\neq q_{\textit{split}}, 
	$$
	and a new alphabet $\Gamma = \left\{a_1,\ldots,a_m\right\}\cup( \Sigma\setminus\left\{a\right\}).$ The CNFA $(Q,\Gamma)$ will be denoted
	$\spl(\mathcal{A},q_{\textit{split}},a)$.
\end{definition}

The idea of this transformation is that we want to make a CNFA $\mathcal{A}$ `more deterministic'. If $|q_{\textit{split}}a|\geq 2$, then multiple outgoing edges in the state $q_{\textit{split}}$ are labelled by the symbol $a$. So we could say that $a$ offers a choice in $q_{\textit{split}}$. For each possible choice, we introduce a new symbol that is deterministic in $q_{\textit{split}}$ and behaves as $a$ in all other states. If $|q_{\textit{split}}a|=1$, then $\mathcal{A}$ is not changed by the transformation. This definition immediately implies the following properties:

\begin{lemma}\label{lemma1} 
	Let $\mathcal{A}=(Q,\Sigma)$ be a CNFA. Fix $q_{\textit{split}}\in Q, a_s\in\Sigma$ and let $(Q,\Gamma)$ be $\spl(\mathcal{A},q_{\textit{split}},a_s)$. Let $c\in T^d(Q)$ be an arbitrary deterministic symbol. Then
	\begin{enumerate}
		\item for all $b\in \Gamma$, there exists $a\in\Sigma$ such that $b\subseteq a$,
		\item (there exists $a\in\Sigma$ such that $c\subseteq a$) $\iff$ (there exists $b\in\Gamma$ such that $c\subseteq b$). 
	\end{enumerate} 
\end{lemma}

\begin{proof} 
	Let $b\in\Gamma$, we will find $a$ with the property claimed in the lemma. If $b\in\Gamma\cap\Sigma$, take $a=b$. If  $b\in\Gamma\setminus\Sigma$, then $b$ is one of the new symbols $a_1,\ldots a_m$. In this case take $a=a_s$. Then $qb = qa$ for $q\neq q_{\textit{split}}$ and $qb\in qa$ for $q=q_{\textit{split}}$. This proves the first statement.
	
	Suppose $a\in\Sigma$ such that $c\subseteq a$, so $qc\in qa$ for all $q$. Then $q_{\it{split}}c\in q_{\it{split}}a$. By Definition \ref{def:parsplit}, there exists $b\in\Gamma$ such that $q_{\it{split}}b= q_{\it{split}}c$ and $qb=qa$ for $q\neq q_{\it{split}}$. This means $c\subseteq b$. Now suppose $b\in\Gamma$ such that $c\subseteq b$. By statement 1 of the lemma, there exists $a\in\Sigma$ such that $b\subseteq a$, which implies $c\subseteq a$.  
\end{proof}

The parametrized $\spl$ preserves synchronization properties, as is shown in the next lemma.

\begin{lemma}\label{lemma2} Let $\mathcal{A}=(Q,\Sigma)$ be a CNFA and let $\mathcal{B} = \spl(\mathcal{A},q_{\textit{split}},a)$ for some $q_{\textit{split}}\in Q, a\in\Sigma$. Then 
	\begin{enumerate}
		\item $\mathcal{A}$ is D3-directing if and only if $\mathcal{B}$ is D3-directing, 
		\item If $\mathcal{A}$ and $\mathcal{B}$ are D3-directing, then $d_3(\mathcal{A}) = d_3(\mathcal{B})$. 
	\end{enumerate}
\end{lemma}
 
\begin{proof} Let $\Gamma$ be the alphabet of $\mathcal{B}$ and denote the new labels by $a_1,\ldots, a_m$. First assume that $\mathcal{A}$ is D3-directing. There exist $w\in\Sigma^*$ and $q_s\in Q$ such that $q_s\in qw$ for all $q\in Q = \left\{q_1,\ldots,q_n\right\}$. From each state $q$ there exists a path labelled by $w=w_1\ldots w_{|w|}$ that ends in $q_s$: 
\begin{eqnarray*}
		q_1 = q_1^0 &\stackrel{w_1}{\longrightarrow}& q_1^1\ \stackrel{w_2}{\longrightarrow}\quad \ldots\quad \stackrel{w_{|w|}}{\longrightarrow}\ q_1^{|w|} = q_s\\
		q_2 = q_2^0 &\stackrel{w_1}{\longrightarrow}& q_2^1\ \stackrel{w_2}{\longrightarrow}\quad \ldots\quad \stackrel{w_{|w|}}{\longrightarrow}\ q_2^{|w|} = q_s\\
		&\vdots&\\
		q_n = q_n^0 &\stackrel{w_1}{\longrightarrow}& q_n^1\ \stackrel{w_2}{\longrightarrow}\quad \ldots\quad \stackrel{w_{|w|}}{\longrightarrow}\ q_n^{|w|} = q_s\\
\end{eqnarray*}
We will construct a word $\tilde w = \tilde w_1\ldots\tilde w_{|w|}\in\Gamma^*$ which follows the same paths. We may assume that paths do not diverge again once they have met, i.e. if $q_i^t = q_j^t$, then $q_i^{t+1} = q_j^{t+1}$. 

Let $Q_t = \left\{q_1^t,\ldots,q_n^t\right\}$ and suppose $w_t=a$ for some $1\leq t\leq |w|$. If $q_{\textit{split}}\not\in Q_{t-1}$, define $\tilde w_t$ to be $a_1$. Then $qw_t = q\tilde w_t$ for all $q\in Q_{t-1}$. If $q_{\textit{split}}\in Q_{t-1}$,  

then $q_i^{t-1} = q_{\textit{split}}$ for some $1\leq i\leq n$. This means $q_i^t\in q_{\textit{split}}a$, so there exists $\tilde a\in\Gamma$ such that $q_i^t=q_{\textit{split}}\tilde a$. Define $\tilde w_t$ to be $\tilde a$. Finally, for all $w_t\neq a$, let $\tilde w_t = w_t$. Then $q_s\in q\tilde w$ for all $q\in Q$, so $\tilde w$ is a D3-directing word for $\mathcal{B}$.

Now assume that $\mathcal{B}$ is D3-directing with D3-directing word $w\in\Gamma^*$ and synchronizing state $q_s$. By repeated application of Lemma \ref{lemma1} (replacing every symbol of $w$ that is not in $\Sigma$ by $a$), it follows that there exists $\tilde w\in\Sigma^*$ such that  that $qw \subseteq q\tilde w$ for all $q$. Therefore $q_s\in q\tilde w$ and the word $\tilde w$ is D3-directing for $\mathcal{A}$.  

The above arguments prove the first statement of the lemma. Clearly, rewriting a D3-directing word from $\Sigma^*$ to $\Gamma^*$ and vice versa preserves the length. This implies the second statement. 
\end{proof}

Next we investigate the result of applying consecutive parametrized $\spl$ transformations to all non-deterministic symbols in a CNFA $\mathcal{A}$. We will show that this terminates and that the result is a uniquely defined DFA.

\begin{lemma}\label{lemma3} Let $\mathcal{A}_0$ be a CNFA, and repeat the following. If $\mathcal{A}_k = (Q,\Sigma_k,\delta)$ is not a DFA, choose $a_k\in\Sigma_k$ and $q_k\in Q$ for which $|q_ka_k|\geq 2$. Let $\mathcal{A}_{k+1} = \spl(\mathcal{A}_k,q_k,a_k)$. Then
	\begin{enumerate}
		\item There exists $k\geq 0$ where this process ends such that $\mathcal{A}_k$ is a DFA.
		\item The resulting DFA does not depend on the choices of $a_k$ and $q_k$. 
	\end{enumerate}
\end{lemma} 

\begin{proof}
	If $\mathcal{A}_k$ is a DFA, then $|qa|=1$ for all $q\in Q$ and $a\in \Sigma_k$. If $\mathcal{A}_k$ is not a DFA, then
	\begin{equation*}
		\prod_{q\in Q}\prod_{a\in \Sigma_k} |qa| > \prod_{q\in Q}\prod_{a\in \Sigma_{k+1}} |qa|\geq 1.
	\end{equation*} 
	As this integer sequence is strictly decreasing, it ends in 1, i.e. there exists $k\geq 0$ such that the $k$th term is equal to 1. This is equivalent to the first claim of the lemma.
	
	For the second claim, choose $k$ such that $\mathcal{A}_k$ is a DFA. Let $c\in T^d(Q)$ be an arbitrary deterministic symbol. By repeated application of the second statement of Lemma \ref{lemma1}, it follows that $c\in\Sigma_k$ if and only if there exists $a\in\Sigma_0$ for which $c\subseteq a$. Therefore the DFA $\mathcal{A}_k$ does not depend on the splitting choices, proving the second statement. 	 	
\end{proof}

\begin{definition}
	Let $\mathcal{A}=(Q,\Sigma)$ be a CNFA. The unique DFA that is produced by repeated application of the parametrized $\spl$ will be called $\spl(\mathcal{A})$.
\end{definition}

Lemma \ref{lemma3} guarantees that $\spl{\mathcal{(A)}}$ is well-defined. Extension of Lemma \ref{lemma1} leads to the following characterization:

\begin{lemma}\label{lemma4} 
	Let $\mathcal{A}=(Q,\Sigma)$ be a CNFA. Denote the DFA $\spl(\mathcal{A})$ by $(Q,\Gamma)$. Let $b\in T^d(Q)$ be an arbitrary deterministic symbol. Then $b\in\Gamma$ if and only if there exists $a\in\Sigma$ such that $b\subseteq a$.
\end{lemma}

\begin{proof}
	If $b\in\Gamma$, repeatedly apply the first statement of Lemma \ref{lemma1}. If there exists $a\in\Sigma$ such that $b\subseteq a$, then repeated application of the second statement of Lemma \ref{lemma1} proves existence of $b'\in\Gamma$ such that $b\subseteq b'$. Since both $b$ and $b'$ are deterministic symbols, $b=b'$, so $b\in\Gamma$.  
\end{proof}
Moreover, we have the following:

\begin{corollary}\label{cor1}
	If $\mathcal{A}=(Q,\Sigma)$ is a D3-directing CNFA, then $d_3(\mathcal{A}) = d(\spl(\mathcal{A}))$.
\end{corollary}

\begin{proof} This is an immediate consequence of Lemma \ref{lemma2}. \end{proof}

Now also the main result of this section is straightforward:

\begin{theorem}\label{theorem:main} The maximal shortest D3-directing word length for DFAs is the same as for CNFAs, i.e. 
	$$\displaystyle d(n)= cd_3(n).$$
\end{theorem}

\begin{proof}
	Since every DFA is also a CNFA and the notions of synchronization and D3-directedness coincide for DFAs, it follows that $d(n)\leq cd_3(n)$. By Corollary \ref{cor1} every CNFA $\mathcal{A}$ has a corresponding DFA $\spl\mathcal{(A)}$ with the same shortest D3-directing word length. Therefore $cd_3(n)\leq d(n)$. 
\end{proof}

This theorem establishes equivalence of Cern\'y's conjecture to Conjecture \ref{conjecture1}. It also implies the following sharpening of the upper bound for $cd_3(n)$:

\begin{corollary}
	$\displaystyle cd_3(n)\leq \frac{n^3-n}{6}.$
\end{corollary}

%
%

\section{Sharper bounds for several classes of CNFAs}

For several classes of DFAs the \v{C}ern\'y
conjecture has been settled, or at least better upper bounds than the cubic one for
the general case have been obtained. If the $\spl$ transform reduces a CNFA to a DFA that belongs to one of these
classes, then as a direct consequence we obtain improved bounds
for the D3-directing length in the CNFA. In this section we present a
couple of results of this type.

The general pattern of the arguments in this section is as follows. First we give the definition of a property $\mathcal{P}$ for DFAs, together with references to the best known upper bound $u_{\mathcal{P}}$ for synchronization lengths in DFAs satisfying $\mathcal{P}$. Then we give a natural extension of $\mathcal{P}$ to the class of CNFAs. Finally, we show that every CNFA $\mathcal{A}$ satisfying $\mathcal{P}$ is reduced to a DFA $\spl\mathcal{(A)}$ satisfying $\mathcal{P}$. Corollary \ref{cor1} then guarantees that the length of the shortest D3-directing word in $\mathcal{A}$ is at most $u_{\mathcal{P}}$.

\subsection{Cyclic automata}

A DFA $\mathcal{A} = (Q,\Sigma)$ is \emph{cyclic} if
one of the letters in $\Sigma$ acts as a cyclic permutation on
$Q$.
\begin{definition}
	A DFA $\mathcal{A} = (Q,\Sigma)$ with $|Q|=n$ is called cyclic if there
	exists $a\in\Sigma$ such that for all $q\in Q$
	$$
	qa^n=q \quad\textrm{and}\quad qa^k \neq q \quad\textrm{for}\quad 1\leq k\leq n-1.
	$$
\end{definition}
Equivalently, the states can be indexed in such a way that $q_ia = q_{i+1}$ for $1\leq i\leq n-1$, and $q_na=q_1$. Examples of cyclic automata include the well-known sequence
$\mathcal{C}_n$ discovered by \v{C}ern\'y.

Dubuc \cite{dubuc} proved that a synchronizing cyclic DFA has a
synchronizing word of length at most $(n-1)^2$, as predicted by
\v{C}ern\'y's conjecture. We define non-deterministic cyclic
automata as follows.
\begin{definition}\label{def:cyclic}
	A CNFA $\mathcal{A} = (Q,\Sigma)$ is called cyclic if there
	exists $a\in\Sigma$ and an indexing of the states such that
	$$
	q_{i+1}\in q_ia\quad\textrm{for}\quad 1\leq i\leq n-1,\quad\textrm{and}\quad q_1\in q_na.
	$$
\end{definition}
 Note that with this definition, a CNFA is cyclic if and only if it is the extension of a cyclic DFA.
\begin{proposition}
	If $\mathcal{A}$ is a D3-directing cyclic CNFA, then $\mathcal{A}$
	has a shortest D3-directing word of length at most $(n-1)^2$.
\end{proposition}
\begin{proof} Denote $\spl\mathcal{(A)}$ by $(Q,\Gamma)$. Choose $a\in\Sigma$ and an indexing of the states as in Definition \ref{def:cyclic}. Define $b$ such that $q_ib = q_{i+1}$ for $1\leq i\leq n-1$, and $q_nb=q_1$. Then $b\subseteq a$ so Lemma \ref{lemma4} gives $b\in\Gamma$. Therefore $\spl\mathcal{(A)}$ is a cyclic DFA and the result follows.\end{proof}

\subsection{One-cluster automata}

A DFA $\mathcal{A} = (Q,\Sigma)$ is called
\emph{one-cluster} if for some letter $a\in\Sigma$, there is only
one cycle (possibly a self-loop) labelled $a$. For every $q\in Q$, the path $qaaa\ldots$ eventually ends in this cycle. One way to formally define this is:
\begin{definition}\label{def:one-clusterDFA}
	A DFA $\mathcal{A} = (Q,\Sigma)$ is called one-cluster if there
	exists $a\in\Sigma$ and $p\in Q$ such that for all $q\in Q$
	$$
	qa^k = p\quad\textrm{for\ some}\quad k\in\mathbb{N}.
	$$
\end{definition}
Note that cyclic DFAs are contained in the class of one-cluster automata.
B\'eal, Berlinkov and Perrin \cite{beal} proved that
in a synchronizing one-cluster DFA, the length of the shortest
synchronizing word is at most $2n^2-7n+7$. We define one-cluster CNFAs in the following way.
\begin{definition}\label{def:one-cluster}
	A CNFA $\mathcal{A} = (Q,\Sigma)$ is called one-cluster if there
	exists $a\in\Sigma$ and $p\in Q$ such that for all $q\in Q$
	$$
	p\in qa^k\quad\textrm{for\ some}\quad k\in\mathbb{N}.
	$$
\end{definition}

With this definition, cyclic CNFAs are a special case of one-cluster CNFAs. Like for the cyclic case, the CNFA is one-cluster if and only if it it an extension of a one-cluster DFA.

\begin{proposition}
	If $\mathcal{A}$ is a D3-directing one-cluster CNFA, then
	$\mathcal{A}$ has a shortest D3-directing word of length at most
	$2n^2-7n+7$.
\end{proposition}

\begin{proof} Let $\spl(\mathcal{A}) = (Q,\Gamma)$. Choose $a$ and $p$ as in Definition \ref{def:one-cluster} and denote the states by $q_1,\ldots,q_n$. For $i=1,\ldots,n$, let $k_i\geq 1$ be the smallest integer such that $p\in q_ia^{k_i}$ (note that this can also be done if $q_i=p$). Choose $q_i'\in q_ia$ such that $p\in q_i'a^{k_i-1}$. Define a symbol $b$ by $q_ib = q_i'$ for all $i$. Then $b\subseteq a$ and $q_ib^{k_i}=p$. By Lemma \ref{lemma4}, $b\in\Gamma$ and therefore $\spl(\mathcal{A})$ is a cyclic DFA. \end{proof} 

\begin{remark}
	To see if a CNFA is one-cluster, it is sufficient to check pairs of states.  A CNFA $\mathcal{A}=(Q,\Sigma)$ is one-cluster if and only if there exists $a\in\Sigma$ such that for any $q,q'\in Q$ there exist $r,s\in\mathbb{N}$ for which $qa^r\cap q'a^s \neq\emptyset$.
\end{remark}

\subsection{Monotonic automata}

\begin{definition} A DFA $\mathcal{A} = (Q,\Sigma)$ is called \emph{monotonic}
	if $Q$ admits a linear order $<$ such that for each
	$a\in\Sigma$ the map $a:Q\rightarrow Q$ preserves the order
	$<$, i.e.
	$$
	qa \leq q'a\quad\text{whenever}\quad q\leq q'.
	$$
\end{definition}

Ananichev and Volkov \cite{ananichev} proved that the length of the
shortest synchronizing word in a synchronizing monotonic DFA is at
most $n-1$. The following definition extends the notion of
monotonicity to non-deterministic automata:

\begin{definition}
	A CNFA $\mathcal{A} = (Q,\Sigma)$ is called monotonic if
	$Q$ admits a linear order $\leq$ such that for each $a\in\Sigma$
	$$
	\max\left\{qa\right\} \leq
	\min\left\{q'a\right\}\quad\text{whenever}\quad q\leq q'.
	$$
\end{definition}

Every DFA contained in a monotonic CNFA is monotonic. But for a CNFA to be monotonic, it is not sufficient that it contains a monotonic DFA. One can easily extend a monotonic DFA to a CNFA $\mathcal{A}$ for which $\spl(\mathcal{A})$ is not monotonic. Just add to the DFA a symbol that sends every state to the full state set $Q$. 

\begin{proposition}
	If $\mathcal{A}$ is a D3-directing monotonic CNFA, then
	$\mathcal{A}$ has a shortest D3-directing word of length at most
	$n-1$.
\end{proposition}

\begin{proof} Let $\mathcal{A} = (Q,\Sigma)$ and $\spl(\mathcal{A}) = (Q,\Gamma)$. Choose $b\in\Gamma$. By Lemma \ref{lemma4} there exists $a\in\Sigma$ such that $b\subseteq a$, i.e. $qb\in
	qa$ for all $q\in Q$. The monotonicity of $\mathcal{A}$ implies
	$$ 
	qb \leq \max\left\{qa\right\} \leq
	\min\left\{q'a\right\}\leq q'b\quad\text{whenever}\quad q\leq q',
	$$
demonstrating that $\spl\mathcal{(A)}$ is a
monotonic DFA. This implies the result. \end{proof}

\subsection{Orientable automata}

A DFA $\mathcal{A} = (Q,\Sigma)$ is called
\emph{orientable} if $Q$ admits a cyclic order
$$
q_1\prec q_2\prec\ldots\prec q_n\prec q_1
$$
on $Q$, such that for each letter $a\in\Sigma$ the sequence
$q_1a,q_2a,\ldots,q_na$ is (after removal of duplicates) a
subsequence of a cyclic shift of $q_1,q_2,\ldots,q_n$. One can think of this as the order of the states on the circle being preserved under $a$. This can equivalently be formalized in terms of a linear order:
\begin{definition}
	A DFA $\mathcal{A} = (Q,\Sigma)$ is called \emph{orientable}
	if $Q$ admits a strict linear order $q_1< q_2<\ldots< q_n$
	such that for each $a\in\Sigma$ at most one of the following
	inequalities is violated:
	$$
	q_1a\leq q_2a\leq \ldots\leq q_na\leq q_1a.
	$$
\end{definition}
Eppstein \cite{E90} proved that orientable automata satisfy the
conjecture of \v{C}ern\'y: if an orientable DFA is synchronizing,
then the shortest synchronizing word has length at most $(n-1)^2$.
\v{C}ern\'y's own examples $\mathcal{C}_n, n = 2,3,\ldots$ are
orientable. We extend the notion of orientability as follows to
the non-deterministic case:
\begin{definition}
	A CNFA $\mathcal{A} = (Q,\Sigma)$ is called
	\emph{orientable} if $Q$ admits a strict linear order
	$q_1< q_2<\ldots< q_n$
	on $Q$ such that for each $a\in\Sigma$ at most one of the following
	inequalities is violated:
	$$
	\max\left\{q_1a\right\}\leq
	\min\left\{q_2a\right\}\leq\ldots\leq \min\left\{q_na\right\}\leq
	\max\left\{q_na\right\}\leq \min\left\{q_1a\right\}.
	$$
\end{definition}
As is the case for DFAs, the class of orientable CNFAs contains the monotonic CNFAs. Like in the monotonic case, every DFA contained in an orientable CNFA is orientable, but not every extension of an orientable DFA is an orientable CNFA.
\begin{proposition} If $\mathcal{A}$ is a D3-directing orientable CNFA, then
	$\mathcal{A}$ has a shortest D3-directing word of length at most
	$(n-1)^2$.
\end{proposition}
\begin{proof} Suppose $\mathcal{A}= (Q,\Sigma)$ is an orientable CNFA. Let
$\spl(\mathcal{A}) = (Q,\Gamma)$ and choose
$b\in\Gamma$. Then by Lemma \ref{lemma4} there exists $a\in\Sigma$ such that $b\subseteq a$, i.e. $q_kb\in
q_ka$ for all $1\leq k\leq n$. It follows that at most one of the
inequalities
$$
q_1b\leq q_2b\leq \ldots\leq q_nb\leq q_1b
$$
is violated so that $\spl\mathcal{(A)}$ is an orientable
DFA. \end{proof}

\subsection{Automata with underlying Eulerian digraph}

Every DFA $\mathcal{A} = (Q,\Sigma)$ has an underlying directed graph (digraph) with vertex set $Q$ and with edges corresponding to actions of elements of $\Sigma$. Formally, $G=(V,E)$ with $V=Q$ and $E = \bigcup_{a\in\Sigma}\left\{(p,q)\in Q^2:pa=q \right\}$, where we consider $E$ to be a multiset, which means that each edge has a multiplicity.

 A digraph is called Eulerian if it is strongly connected and all indegrees and outdegrees are the same. Kari \cite{kari} proved that a synchronizing DFA for which the underlying graph is Eulerian admits a synchronizing word of length at most $(n-2)(n-1)+1$.

For a CNFA $\mathcal{A}$ with underlying Eulerian digraph $G$ it is not necessarily the case that the underlying digraph of the DFA $\spl\mathcal{(A)}$ is Eulerian as well. However, we can decompose $G$ into directed graphs for each of the symbols in $\Sigma$: define $G_a=(V,E_a)$ with $V=Q$ and $E_a = \left\{(p,q)\in Q^2:pa=q\right\}$. The following stronger property is preserved under $\spl$:
\begin{definition}
	A CNFA $\mathcal{A}=(Q,\Sigma)$ is called strongly Eulerian if for all $a\in\Sigma$ the corresponding digraph $G_a$ is Eulerian.
\end{definition}
\begin{proposition} If $\mathcal{A}=(Q,\Sigma)$ is a D3-directing strongly Eulerian CNFA, then $\mathcal{A}$ has a shortest D3-directing word of length at most $(n-2)(n-1)+1$.
\end{proposition}
\begin{proof} Let $\spl(\mathcal{A})=(Q,\Gamma)$ and denote its underlying digraph by $H$. Let's first assume that $\Sigma$ is a singleton $\left\{a\right\}$ and that all in- and outdegrees of $G_a$ are equal to $k$. Suppose $q,q'\in Q$ are such that $q'\in qa$. Since all $p\neq q$ have $a$-outdegree $k$, there will be exactly $k^{n-1}$ letters $\tilde a_i\in\Gamma, i = 1,\ldots,k^{n-1}$ for which $q' = q\tilde a_i$. Doing this for all $q'\in qa$ and using that $|qa|=k$ we obtain that $q$ has outdegree $k^n$ in $H$. An analogous argument gives the same for the indegrees.

Now suppose $\Sigma$ is not a singleton, i.e. $\Sigma = \left\{a_1,\ldots,a_m\right\}$. Assume that all degrees in $G_{a_i}$ are equal to $k_i$. Then by repeatedly applying the above argument, we obtain that all degrees in $H$ are equal to $\sum_{i=1}^m k_i^n$. Clearly $H$ is also strongly connected. Therefore $H$ is Eulerian which implies the result.\end{proof}

\subsection{Aperiodic automata}

One way to define aperiodic DFAs is the following:

\begin{definition}
	A DFA $\mathcal{A}=(Q,\Sigma)$ is called aperiodic if for all $w\in\Sigma^*$ and $q\in Q$ there exists $k\geq 0$ such that $qw^k = qw^{k+1}$.
\end{definition}

An aperiodic DFA with strongly connected underlying digraph is synchronizing and has a synchronizing word of length at most $\lfloor\frac{n(n+1)}{6}\rfloor$, see \cite{volkov2}.

The definition could also be written down for CNFAs, but this property is not preserved under the $\spl$ transformation. For example, let $\mathcal{A}=(Q,\Sigma)$ with $|Q|\geq 2$ and $\Sigma = \left\{a\right\}$, where $a$ is defined by $qa = Q$ for all $q\in Q$. Then clearly $qw^k = qw^{k+1}$ for all $w\in\Sigma^*$ and $q\in Q$. We will show that $\mathcal{B}=\spl(\mathcal{A}) = (Q,\Gamma)$ admits periodic words. Let $q_1,q_2\in Q, q_1\neq q_2$ and let $b\in T^d(Q)$ be such that $q_1b=q_2$ and $q_2b=q_1$. Then $b\subseteq a$ (as $a$ contains every possible symbol), so by Lemma \ref{lemma4} it follows that $b\in \Gamma$. However, $q_1b^k = q_1$ if $k$ even and $q_1b^k = q_2$ if $k$ odd. Therefore $\mathcal{B}$ fails to be aperiodic.

\begin{proposition}\label{prop7} Suppose $\mathcal{A}=(Q,\Sigma)$ is a CNFA with the following property: for all $w\in\Sigma^*$ and $q\in Q$ there exists $k\geq 0$ such that $qw^k = qw^{k+1}$ and $|qw^k|=1$. If its underlying digraph is strongly connected, then  $\mathcal{A}$ is D3-directing and has a D3-directing word of length at most $\lfloor\frac{n(n+1)}{6}\rfloor$.
\end{proposition}

\begin{proof}
	Let $\mathcal{B}=\spl(\mathcal{A}) = (Q,\Gamma)$. Let $w = w_1w_2\ldots w_l\in\Gamma^*$. By Lemma \ref{lemma4} there exist $v = v_1v_2\ldots v_l\in\Sigma^*$ such that $w_i\subseteq v_i$ for all $i$. Let $q\in Q$ and choose $k$ such that $qv^k=qv^{k+1}$ and $|qv^k|=1$. Since $q_w^k$ is not empty and $qw^k\subseteq qv^k$ it follows that $qw^k = qv^k$ and similarly $qw^{k+1} = qv^{k+1}$. Consequently, $\mathcal{B}$ is an aperiodic DFA. Furthermore, if $\mathcal{A}$ has a strongly connected underlying digraph, then so has $\mathcal{B}$.  
\end{proof} 

CNFAs with the property of Proposition \ref{prop7} transform by $\spl$ into an aperiodic DFA. However, having this property is not necessary for being transformed into an aperiodic DFA, as the next example shows. Let $\mathcal{A}=(\left\{1,2 \right\},\left\{a\right\})$ where $a$ is defined by $1a = \left\{1,2\right\}$ and $2a = 2$. Then $|1a^k| = 2$ for all $k\geq 1$. Nevertheless $\spl(\mathcal{A})$ is aperiodic, as can be easily verified.       

\section{Investigating critical CNFAs}
In \cite{DZ17} all critical DFAs on 3 or 4 states were identified, and for $n \geq 5$ states all critical 
extensions of known critical DFAs. Recently it was confirmed \cite{BDZ17} that for $n=5$ and 6 no more critical DFAs exist.
From Corollary \ref{cor1} we know that $\spl(\mathcal{A})$ is a critical DFA for every critical
CNFA $\mathcal{A}$. So if for any DFA $\mathcal{D}$ we can investigate which CNFAs map to
$\mathcal{D}$ by $\spl$, we can combine these observations to identify all critical CNFAs with $<7$ states. For $\geq 7$ states the investigation restricts to resulting known critical DFAs. In doing so, first 
we concentrate on investigating which CNFAs 
map by $\spl$ to a given DFA, independent of size or being critical.  

For a DFA $\mathcal{D} = (Q,\Sigma)$ we define a graph structure on $\Sigma$. More
precisely, we define $G(\mathcal{D}) = (\Sigma,E)$ to be the undirected graph of which 
$\Sigma$ is the set of nodes and the set $E$ of edges is defined by
\[ \{a,b\} \in E \Longleftrightarrow \exists q \in Q: qa \neq qb \wedge \forall r \neq q : ra = rb.\]
Next we show that if $G(\mathcal{D}) = (\Sigma,E)$ then any $E' \subseteq E$ gives rise to a
CNFA $N(\mathcal{D},E')$ such that $\spl(N(\mathcal{D},E')) = \mathcal{D}$.
The CNFA $N(\mathcal{D},E') = (Q,\Sigma')$ is defined by 
\[ \Sigma' = \{a \cup b \mid \{a,b\} \in E'\} \cup \{a \mid \not\exists b : \{a,b\} \in E'\}.\]
So symbols not connected by an edge in $E'$ remain unchanged, and any two symbols $a,b$ that are 
connected by an edge in $E'$ are joined into the new symbol $a\cup b$. It is defined by $q(a \cup b) = 
\{qa,qb\}$ for the single state $q$ with $qa \neq qb$, and $r(a \cup b) = ra = rb$ for $r \neq
q$. In particular, every symbol in $N(\mathcal{D},E')$ is
non-deterministic in at most one state, and in that state only two choices are possible. 

As an example consider the following DFA $A_3^+$ on three states, and six symbols $a,b,c,d,e,f$.
In fact it is the DFA $A_3$ from \cite{DZ17}, extended by an extra symbol $f$ that acts as the
identity. On the right its graph $G(A_3^+)$ is shown: the nodes are $a,b,c,d,e,f$, and there are
three edges $\left\{b,c\right\}$, $\left\{b,f\right\}$ and $\left\{d,e\right\}$. These are exactly the pairs of symbols acting in the same way on two of the
three states.

\vspace{2mm}

\includegraphics[scale=0.25]{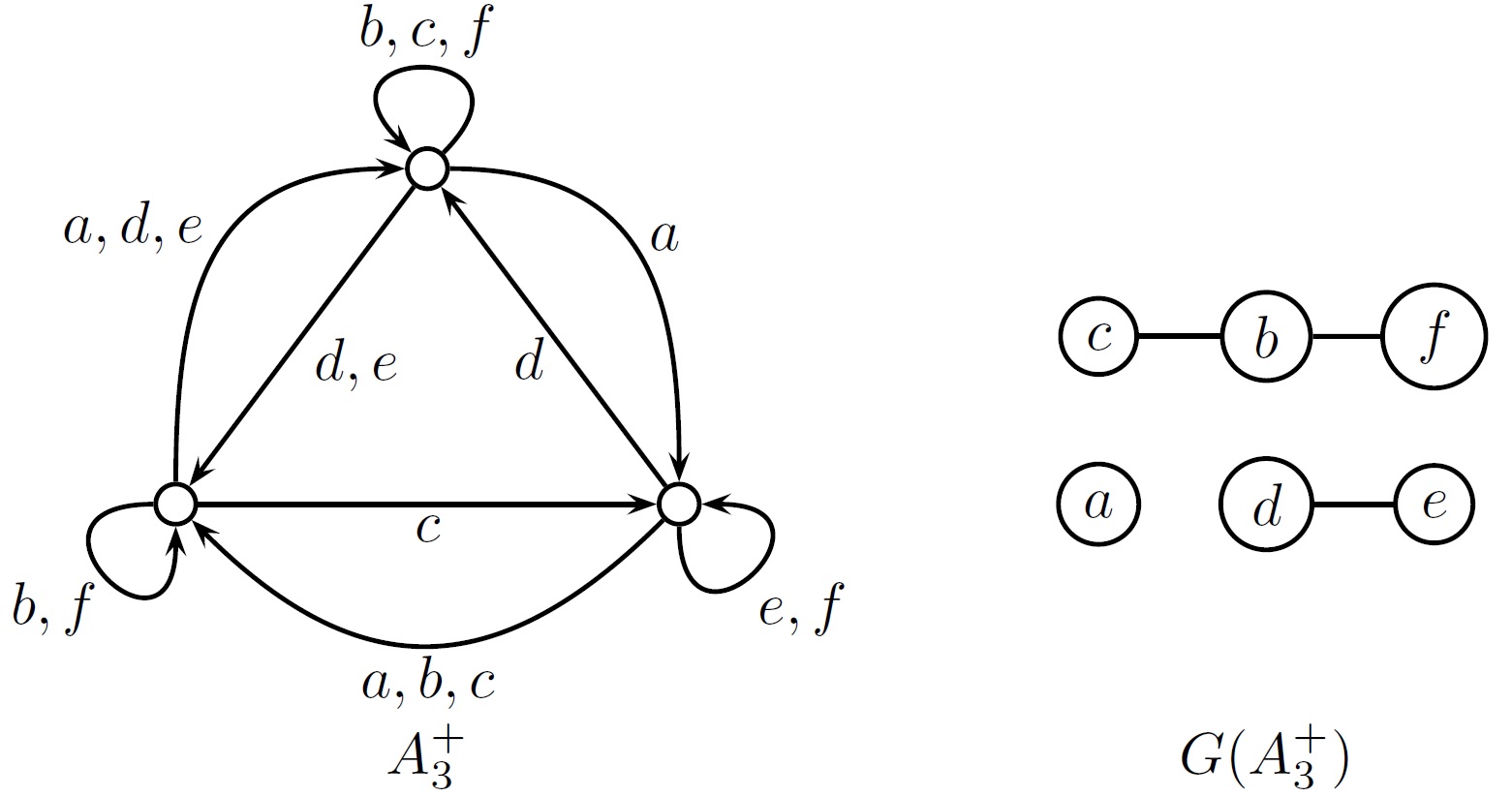}

\vspace{2mm}

The edge set $E$ of the graph $G(A_3^+)$ has size 3, so there are 8 possible choices for $E'$.
As an example we show the resulting CNFA $N(\mathcal{D},E')$ for
$E' = \{\{b,c\},\{d,e\}\}$, in which there are 4 symbols $a, b \!\cup\! c, d \!\cup\! e, f$:

\vspace{2mm}

\includegraphics[scale=0.25]{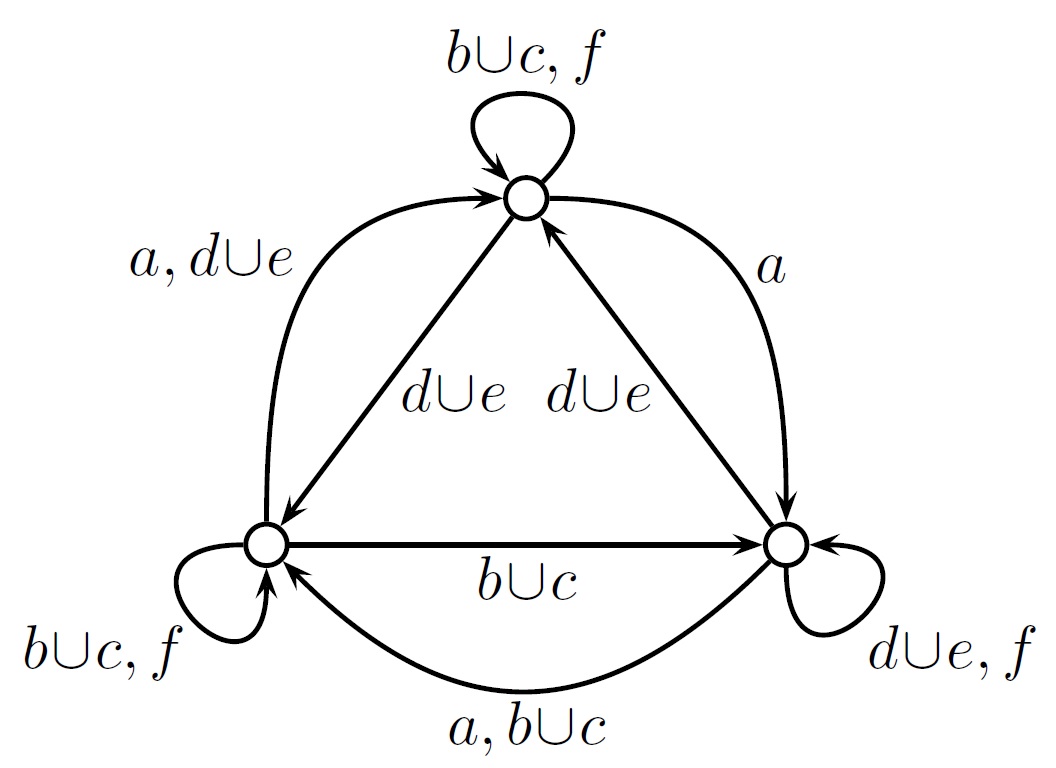}

\vspace{2mm}

The next theorem states that under some conditions for a given DFA $\mathcal{D}$ these CNFAs
$N = N(\mathcal{D},E')$ are exactly all {\em pre-basic} CNFAs $N$ for which  $\spl(N) = \mathcal{D}$.
Here a CNFA is called {\em pre-basic} if no symbol is contained in another; the difference 
with basic is that now the identity symbol is allowed.
Note that by definition every CNFA of the shape $N(\mathcal{D},E')$ is pre-basic, and also every DFA is
pre-basic.

\begin{theorem}
	\label{thmsplitinv}
	Let $\mathcal{D} = (Q,\Sigma)$ be a DFA for which $G(\mathcal{D}) = (\Sigma,E)$ does not admit cycles
	of length 3 or 4.
	Then a pre-basic CNFA $\mathcal{N} = (Q,\Sigma')$ satisfies $\spl(\mathcal{N}) = \mathcal{D}$
	if and only if $\mathcal{N} = N(\mathcal{D},E')$ for some $E' \subseteq E$.
\end{theorem}
\begin{proof}
	For the `if'-part, let $\mathcal{N} = N(\mathcal{D},E')$ for $E' \subseteq E$.
	If $E' = \emptyset$ then $\mathcal{N} = \mathcal{D}$, so $\spl(\mathcal{N}) = \mathcal{D}$.
	Otherwise, let $\{a,b\} \in E'$ and $E'' = E' \setminus \{\{a,b\}\}$. 
	Since $\{a,b\} \in E$ there exists $q \in Q$ such that $qa \neq qb$ and $ra = rb$ for all $r \neq q$.
	Now it is straightforward from the definitions that $\spl(N(\mathcal{D},E'),q,a) =
	N(\mathcal{D},E'')$. By repeating this process by removing all elements from $E'$ one by one, 
	after applying a number of $\spl$ operations on $\mathcal{N} = N(\mathcal{D},E')$ we obtain
	$N(\mathcal{D},\emptyset) = \mathcal{D}$, hence  $\spl(\mathcal{N}) = \mathcal{D}$.
	
	For the `only if'-part we have to show that no other pre-basic CNFA $\mathcal{N}$ satisfies
	$\spl(\mathcal{N}) = \mathcal{D}$.
	So let $\mathcal{N} = (Q,\Sigma')$ be an arbitrary CNFA satisfying $\spl(\mathcal{N}) = \mathcal{D}  = (Q,\Sigma)$.
	We will prove by induction on the length of the shortest $\spl(\mathcal{N},-,-)$-path from 
	$\mathcal{N} = (Q,\Sigma')$ to $\mathcal{D}$ that $\mathcal{N} = N(\mathcal{D},E')$ for some 
	$E' \subseteq E$. Let the first step be $\spl(\mathcal{N},q,a) = \mathcal{N}'$ in which
	$qa$  consists of at least two states for $q \in Q, a \in \Sigma'$; by the
	induction hypothesis we assume $\mathcal{N}' = N(\mathcal{D},E'')$ for some 
	$E'' \subseteq E$. 

        First assume that $qa$ consists of at least three states. 
	Then according to Lemma \ref{lemma4} every $b \in T^d(Q)$ satisfying $b \subseteq a$ is in $\Sigma$.
	Among these there are three symbols $b_1,b_2,b_3 \in \Sigma$ such that 
	the states $q b_i$ are all distinct for $i=1,2,3$ and for every $r \neq q$ the 
	states $r b_i$ are all equal for $i=1,2,3$.  
	But then $b_1,b_2,b_3$ form a 3-cycle in $G(\mathcal{D})$, contradicting the assumption of
	the theorem. 

Hence $qa$ consists of exactly two states $qa_1,qa_2$, where $a_1,a_2$ are symbols in
$\spl(\mathcal{N},q,a) = \mathcal{N}'$ for which $qa_1,qa_2$ are single states and $ra_1 = ra_2 = ra$ for
all other states $r$. Now we claim that $ra$ consists of exactly one state 
for every $r \neq q$. If not, then choose $r \in Q$ for which $r_1, r_2 \in ra$, $r_1 \neq r_2$. For all
other states $s$ choose $s' \in sa$. For $i,j = 1,2$ define $q a_{i,j} = a_i$, $r a_{i,j} = r_j$ and
$s a_{i,j} = s'$ for all other states $s$. Then for $i,j = 1,2$ we obtain $a_{i,j} \subset a_i$, so 
$a_{i,j} \in \Sigma$ by Lemma \ref{lemma4}. But this yields a 4-cycle $a_{1,1}, a_{1,2}, a_{2,2}, a_{2,1}, a_{1,1}$ 
in $G(\mathcal{D})$, contradicting the assumption of the theorem. Hence indeed $ra$ consists of exactly one state
for every $r \neq q$. Hence $a_1,a_2 \in  T^d(Q)$, and $a = a_1 \cup a_2$, and $\{a_1,a_2\} \in E$.
Since $\spl(\mathcal{N},q,a) = \mathcal{N}' = N(\mathcal{D},E'')$, the non-deterministic symbols of both 
$\mathcal{N}$ and $N(\mathcal{D},E'' \cup \{\{a_1,a_2\}\})$ are
exactly $a = a_1\cup a_2$ and the non-deterministic symbols of $\mathcal{N}'$. Since $\spl(\mathcal{N}) =
\mathcal{D}  = (Q,\Sigma)$ and $\mathcal{N}$ is pre-basic, the deterministic symbols of $\mathcal{N}$ 
are exactly the symbols from $\Sigma$ that are not covered by $E'' \cup \{\{a_1,a_2\}\}$. As the same holds
for $N(\mathcal{D},E'' \cup \{\{a_1,a_2\}\})$, we conclude that all symbols of $\mathcal{N}$ and
$N(\mathcal{D},E'' \cup \{\{a_1,a_2\}\})$ coincide. Hence $\mathcal{N} = N(\mathcal{D},E'' \cup
\{\{a_1,a_2\}\})$, concluding the proof.
\end{proof} 

The proof demonstrates that the requirement concerning cycles is only needed for one of the implications in Theorem \ref{thmsplitinv}. If $G(\mathcal{D})$ does contain a 3- or 4-cycle, it is still possible to detect CNFAs that are mapped to $\mathcal{D}$ by $\spl$. As before, every set of edges in $G(\mathcal{D})$ corresponds to such a CNFA. However, there might exist other CNFAs which are mapped to $\mathcal{D}$ as well. 

The following examples show that for the other implication in Theorem \ref{thmsplitinv} it is essential to disallow cycles of length 
both 3 and 4 in $G(\mathcal{D})$. Let $\mathcal{N}$ be defined by $Q = \{1,2,3\}$, $\Sigma = \{a\}$, 
$1a = \{1,2,3\}$, $2a = \{2\}$, $3a = \{3\}$. Then in $\mathcal{D} = \spl(\mathcal{N}) = \spl(\mathcal{N},1,a)$
we have three symbols that form a 3-cycle in $G(\mathcal{D})$, and  $\mathcal{N}$ is not of the shape
$N(\mathcal{D},E')$ for some set $E'$ of edges of $G(\mathcal{D})$.

As a next example let $\mathcal{N}$ be defined by $Q = \{1,2\}$, $\Sigma = \{a\}$, 
$1a = 2a = \{1,2\}$. Then in $\mathcal{D} = \spl(\mathcal{N})$
we have four symbols that form a 4-cycle in $G(\mathcal{D})$, and  $\mathcal{N}$ is not of the shape
$N(\mathcal{D},E')$ for some set $E'$ of edges of $G(\mathcal{D})$.

The role of being pre-basic is illustrated in the following example. Let $\mathcal{N}$ be defined by 
$Q = \{1,2\}$, $\Sigma = \{a,b\}$, 
$1a = \{1,2\}, 2a = 1b = 2b = \{2\}$. Then  $\mathcal{N}$ is not pre-basic since $b \subseteq a$. In 
$\mathcal{D} = \spl(\mathcal{N},1,a)$ the symbol $a$ is split into two symbols $a_1,a_2$, in which $a_2 = b$.
So by $\spl$ the symbol $b$ disappears, and in $N(\mathcal{D},\{\{a_1,a_2\}\})$ there is only one symbol $a$.

Note that if $\mathcal{D}$ is basic, then $N(\mathcal{D},E')$ is basic
too. However, the converse does not hold. For instance, if $\Sigma = \{a\}$, $Q = \{1,2\}$, 
$1a = \{1,2\}$, $2a = \{2\}$, then the CNFA $\mathcal{N}  = (Q,\Sigma)$ is basic, but 
$\spl(\mathcal{N})$ is not since one of the symbols acts as the identity.

As in \cite{DZ17} basic critical DFAs were investigated, we want to combine this by 
Theorem \ref{thmsplitinv} to investigate basic critical CNFAs. For doing so, we need the following  
lemma. For any CNFA $\mathcal{N} = (Q,\Sigma)$ we write $\mathcal{N}^+$ for $(Q,\Sigma \cup \{\id\})$
for $\id$ being the identity function on $Q$.

\begin{lemma}
\label{lempreb}
A basic CNFA $\mathcal{N}$ is critical if and only if $\spl(\mathcal{N}^+) = \mathcal{D}^+$ for some 
basic critical DFA $\mathcal{D}$.
\end{lemma}
\begin{proof}
By Corollary  \ref{cor1} and the fact that adding $\id$ does not influence synchronization, the CNFA 
$\mathcal{N}$ is critical if and only if $\spl(\mathcal{N}^+)$ is critical. 
The 'if'-part follows since $\mathcal{D}^+$ is critical if  $\mathcal{D}$ is critical.
For the 'only if'-part let $\mathcal{N}$ be a basic critical CNFA. Then $\spl(\mathcal{N}^+)$ is 
a critical DFA. Since $\id$ is contained in $\mathcal{N}^+$, by Lemma \ref{lemma4} we obtain that $\id$ 
is also contained in 
$\spl(\mathcal{N}^+)$. Hence $\spl(\mathcal{N}^+) = \mathcal{D}^+$ for some basic DFA $\mathcal{D}$, which is
critical since $\spl(\mathcal{N}^+) = \mathcal{D}^+$ is critical.
\end{proof}

So basic critical CNFAs can be obtained by taking a basic critical DFA
$\mathcal{D}$ and compute $G(\mathcal{D}^+)$. If it does not contain 3- or 4-cycles,
all pre-basic CNFAs $\mathcal{N}$ with $\spl(\mathcal{N}) = \mathcal{D}^+$ can be
obtained by Theorem \ref{thmsplitinv}. Finally, by Lemma \ref{lempreb} the resulting 
basic critical CNFAs are obtained by removing $\id$ wherever it occurs.
As long as no  3- or 4-cycles occur in $G(\mathcal{D}^+)$, all basic critical CNFAs
can be obtained in this way. The number of resulting basic critical CNFAs is equal to
the number of pre-basic critical CNFAs obtained by Theorem \ref{thmsplitinv} since 
by removing the possible occurrence of $\id$ any pre-basic CNFA is transformed to a basic CNFA, and no two
map to the same basic CNFA since for a basic CNFA $\mathcal{N}$ with $\spl(\mathcal{N}) = \mathcal{D}^+$, 
the CNFA $\mathcal{N}^+$ is not pre-basic since it has a symbol in which $\id$ is contained.

For instance, in the cyclic graph $G(A_3^+)$ we have $f = \id$, and there are exactly 3 edges, yielding 
exactly 8 sets of edges. Hence there are exactly 8 basic CNFAs $\mathcal{N}$ with $\spl(\mathcal{N}^+) =
A_3^+$, all obtained from $N(A_3^+,E')$ for a set $E'$ of edges. If $\{b,f\} \in E'$ then this CNFA is already
basic since the identity $f = \id$ is joined with $b$, and otherwise the symbol $f = \id$ is removed.

For $n \leq 6$ all critical DFAs are known; next for all the $n \leq 6$ we apply the above approach to find  and
count all basic critical CNFAs.

\subsection{Analyzing 2 states}

In \cite{DZ17} no analysis of automata with two states was made, since that is a quite
degenerate case in which the maximal shortest synchronizing word is only one single symbol. On two
states $1,2$ there are three possible deterministic non-identity symbols: $a$ mapping both states to 1,  
$b$ mapping both states to 2, and $s$ swapping 1 and 2. A DFA is critical if and only if at least one of the symbols $a$ and $b$ occurs. These yield exactly 6 basic critical
DFAs, having alphabets $\{a\}, \{b\}, \{a,b\},\{a,s\}, \{b,s\}, \{a,b,s\}$, among which $\{a\}$ is isomorphic
to $\{b\}$ and $\{a,s\}$ is isomorphic to $\{b,s\}$, so up to isomorphism there are exactly 4. To each of these DFAs, we add $f = \id$ and apply the above approach.
Computing the graph and counting the edges gives 
the following numbers of basic critical CNFAs corresponding to sets of edges 
\[ \begin{array}{|c|c||c|c|}
\hline
\mbox{symbols DFA} & \mbox{nr of CNFAs} & \mbox{symbols DFA} & \mbox{nr of CNFAs} \\
\hline
a,b,s & 2^4 = 16 & a,b & 2^2 = 4 \\
a,s & 2^2 = 4 & a & 2^1 = 2 \\
b,s & 2^2 = 4 & b & 2^1 = 2 \\
\hline 
\end{array} \]
so yielding 32 CNFAs corresponding to sets of edges. The DFA $\mathcal{D}$ with all three symbols $a,b,s$ together with $\id$ yields a graph that is a 4-cycle:
the edges are $(a,s)$, $(b,s)$, $(a,\id)$, $(b,\id)$. Theorem \ref{thmsplitinv} does not apply to $\mathcal{D}$ due to this 4-cycle, which means that there might be more CNFAs that are mapped to $\mathcal{D}$ by $\spl$. It turns out that apart from the 16 in the table there is one more CNFA $\mathcal{N}$ for which $\spl(\mathcal{N}) = \mathcal{D}$: the CNFA on $\{1,2\}$ with one symbol $c$ satisfying $1c = 2c =
\{1,2\}$. This makes 33 basic critical CNFAs on two states in total; up to isomorphism the
number is 20.

\subsection{Analyzing 3 states}
From \cite{DZ17} we recall that there are exactly 15 basic critical DFAs on 3 states, namely 
the restrictions of $A_3$ to the 15 sets of symbols indicated in the table below.
For each of them we add $f = \id$ and proceed as before. Since $G(A_3^+)$ does not
contain 3- or 4-cycles, also its subgraphs don't. Therefore Theorem \ref{thmsplitinv} applies and by the above
approach we obtain the following numbers $2^k$ of basic critical CNFAs in which $k$ is the number of edges in the
corresponding graph.
\[ \begin{array}{|c|c||c|c|}
\hline
\mbox{symbols DFA} & \mbox{nr of CNFAs} & \mbox{symbols DFA} & \mbox{nr of CNFAs} \\
\hline
a,b,c,d,e & 2^3 = 8 & a,b,e & 2^1 = 2 \\
a,b,c,d & 2^2 = 4 & a,c,d & 2^0 = 1 \\
a,b,c,e & 2^2 = 4 & a,d,e & 2^1 = 2 \\
a,b,d,e & 2^2 = 4 & b,c,e & 2^2 = 4 \\
a,c,d,e & 2^1 = 2 & c,d,e & 2^1 = 2 \\
b,c,d,e & 2^3 = 8 & a,b & 2^1 = 2 \\
a,b,c & 2^2 = 4 & a,d & 2^0 = 1 \\
a,b,d & 2^1 = 2 &  & \\
\hline 
\end{array} \]

So we conclude that there are exactly 50 basic critical CNFAs on three states, including the 15 basic critical
DFAs, corresponding to the empty set of edges. No two of them are isomorphic.

\subsection{Analyzing 4 states}
From \cite{DZ17} we recall that there are exactly 12 basic critical DFAs on 4 states, namely $C_4$ and
T4-2 depicted below and the restrictions of $A_4$ to the 10 sets of symbols
$\{a,b,c,d,e\}$, $\{a,b,c,d\}$, $\{a,b,c,e\}$, $\{a,b,d,e\}$, $\{b,c,d,e\}$, $\{a,b,c\}$, $\{a,b,d\}$,
$\{a,b,e\}$, $\{b,d,e\}$ and $\{a,b\}$. 
Here $A_4$ is the restriction of  $A_4^+$ to $abcde$ depicted
below. In fact, $A_4^+$ was obtained from A4 by adding an extra symbol $f = \id$.

\vspace{2mm}

\includegraphics[scale=0.22]{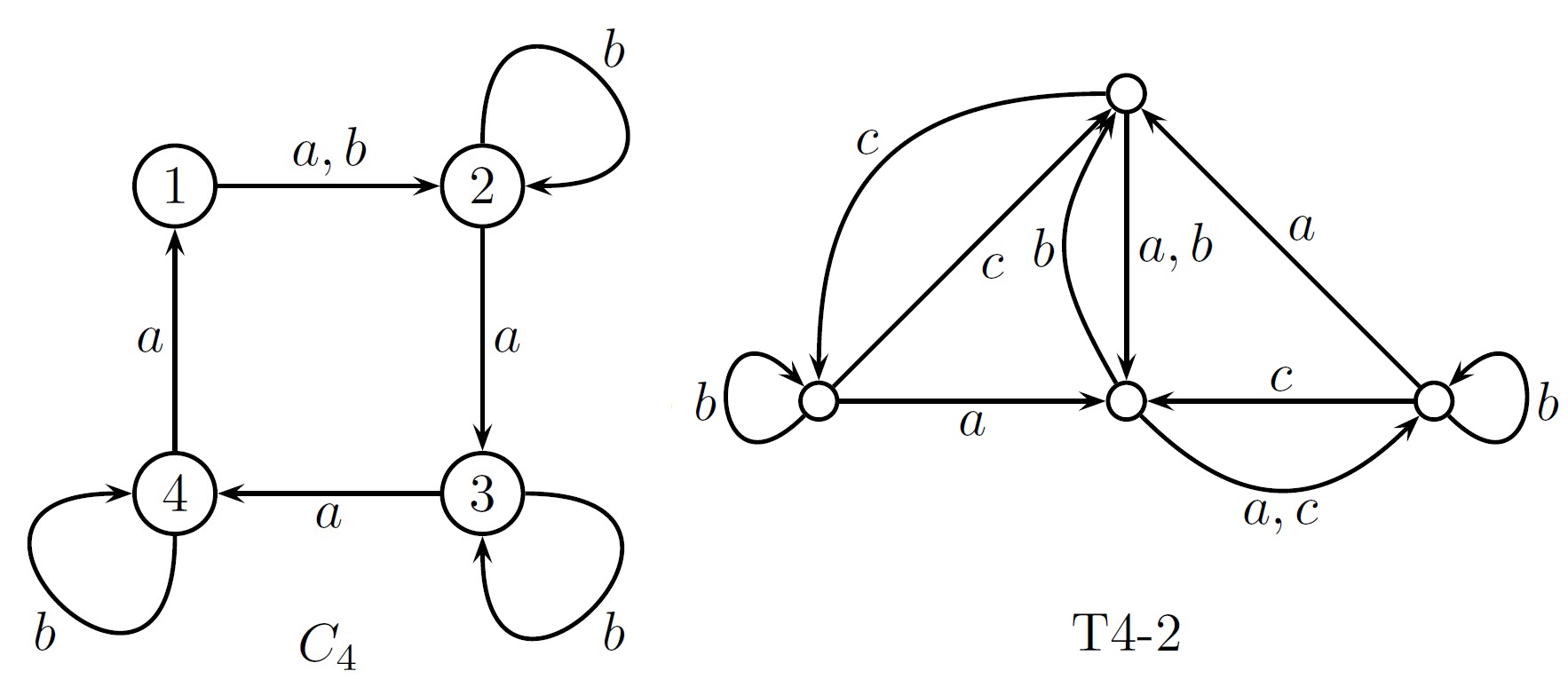}

\vspace{2mm}

\includegraphics[scale=0.3]{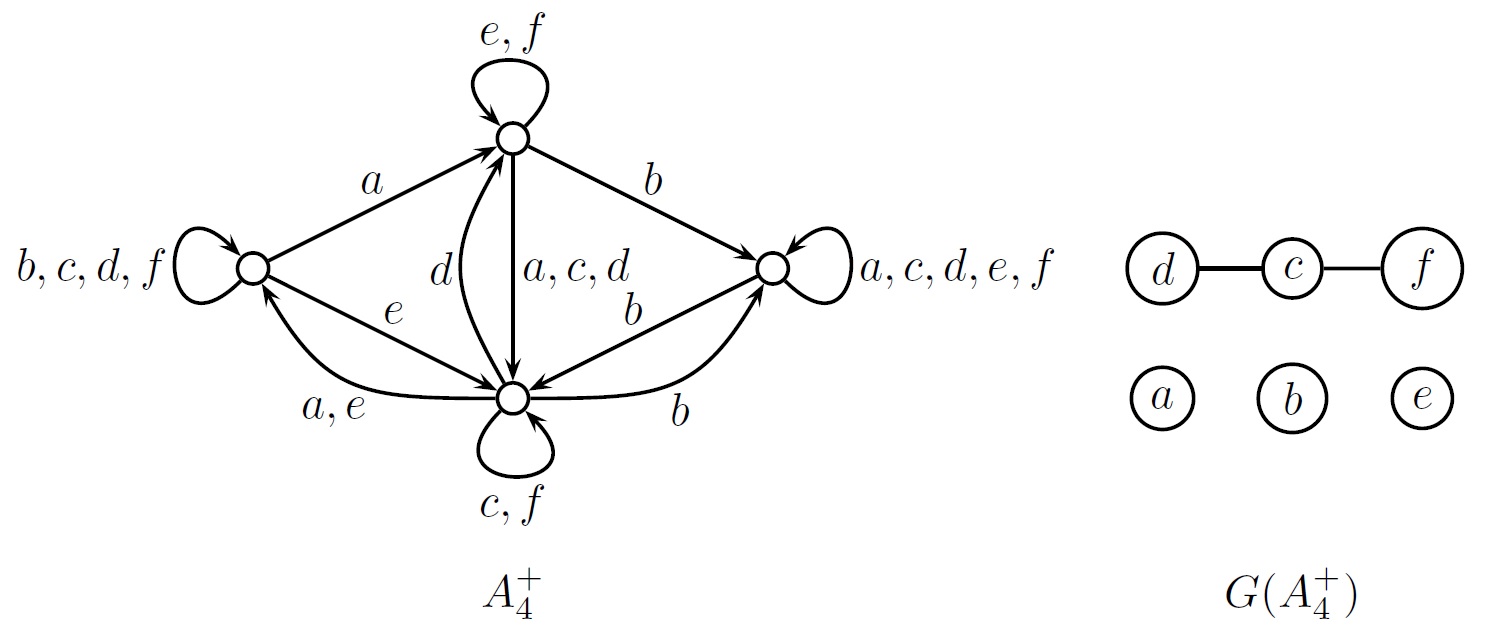}

\vspace{2mm}

It is straightforward that $G(C_4^+)$, is a graph with one edge: $b$ connected to $\id$, hence
yielding exactly two basic CNFAs.  Doing the same for 
$\mbox{T4-2}^+$ yields the empty graph, so there are no more CNFAs for which the
$\spl$ is T4-2 then only T4-2 itself.

For $A_4^+$ the corresponding graph $G(A_4^+)$ is indicated, so on all its subgraphs  Theorem
\ref{thmsplitinv} applies. For the 10 relevant subgraphs we obtain the following numbers of basic critical CNFAs.
\[ \begin{array}{|c|c||c|c|}
\hline
\mbox{symbols DFA} & \mbox{nr of CNFAs} & \mbox{symbols DFA} & \mbox{nr of CNFAs} \\
\hline
a,b,c,d,e & 2^2 = 4 & a,b,c & 2^1 = 2  \\
a,b,c,d & 2^2 = 4  & a,b,d &  2^0 = 1 \\
a,b,c,e & 2^1 = 2 & a,b,e &  2^0 = 1 \\
a,b,d,e & 2^0 = 1 & b,d,e &  2^0 = 1 \\
b,c,d,e & 2^2 = 4  & a,b &  2^0 = 1 \\
\hline 
\end{array} \]

So we conclude that there are exactly 21 basic critical CNFAs $\mathcal{N}$ for which the $\spl(\mathcal{N}^+)$ is 
$A_4^+$. Together with
the two related to $C_4$ and the single one T4-2 this yields exactly 24 basic critical CNFAs on four states (none of them isomorphic),
including the 12 basic critical DFAs.

\subsection{Analyzing $\geq 5$ states}

On $\geq 5$ states the only known basic critical DFAs are $C_n$ for every $n \geq 5$, and two
more: one from Roman on 5 states and 3 symbols, and 
one from Kari on 6 states and 2 symbols, for their definitions and references we refer to \cite{DZ17}.
It was shown by Trahtman \cite{T06} that this investigation is complete for $n \leq 10$ when
restricting to at most two symbols; recently in \cite{BDZ17} it was shown that for 
$n \leq 6$ indeed no more critical DFAs exists, with no restrictions on the numbers of symbols. 

For both the Kari and Roman DFA the corresponding graph is empty, yielding no other basic
critical CNFAs. The DFA $C_n$ is defined to be \linebreak $(\{1,2,\ldots,n\},\{a,b\})$ for $a,b$ defined by
$qa = q+1$ for $q < n$, $na = 1$, $1b = 2$, $qb = q$ for $q > 1$. 
After adding the identity $\id$ its graph consists of a single edge $\{b,\id\}$, yielding one more basic critical 
CNFA on $n$ states: $C_n$ to which a $b$-self-loop is added to 1, yielding $1b = \{1,2\}$.
Summarizing, for $n \leq 6$ we have up to isomorphism the following numbers of basic critical DFAs and CNFAs:
\[ \begin{array}{|c|c|c|}
\hline
\mbox{nr of states} & \mbox{nr of DFAs} & \mbox{nr of CNFAs} \\
\hline
2 & 4 & 20 \\
3 & 15 & 50 \\
4 & 12 & 24 \\
5 & 2 & 3 \\
6 & 2 & 3 \\
\hline 
\end{array} \]
while for $n > 6$ the only known basic critical DFA is $C_n$, to be extended to exactly one more basic critical CNFA.

\section{Conclusions}

The central result of this paper is that every D3-directing CNFA can be transformed to a synchronizing DFA with the same synchronizing word length. In this paper we present this $\spl$ transformation and explore its properties. An immediate consequence is that the maximal shortest D3-directing length for CNFAs is equal to the maximal shortest synchronizing length for DFAs, which means that the famous \v{C}ern\'y conjecture for DFAs extends to CNFAs. 

For several classes of DFAs with some additional properties, tighter bounds for synchronization lengths have been established. If a CNFA is transformed into a DFA belonging to such a class, then the tighter bound also applies to the CNFA. This observation is used to define properties for CNFAs that guarantee improvements over the general cubic bound.  

In the last part of the paper, we investigate critical CNFAs. The tight connection between critical DFAs and critical CNFAs, combined with the fact that critical DFAs are extremely rare, implies that also a very small fraction of CNFAs is critical. All critical DFAs on at most 6 states are known. By essentially inverting $\spl$, we identify all critical CNFAs on at most 6 states. Furthermore, for all $n\geq 3$ there is exactly one critical CNFA which is a strict extension of \v{C}ern\'y's DFA $C_n$.

\bibliographystyle{plain}
\bibliography{ref}

\end{document}